\documentclass[letterpaper, 10pt, conference]{ieeeconf}

\IEEEoverridecommandlockouts                           

\usepackage{amsmath,amsfonts, amssymb,color}
\usepackage{tikz, subcaption, cite}

\newtheorem{theorem}{Theorem}
\newtheorem{corollary}{Corollary}

\newtheorem{proposition}{Proposition}
\newtheorem{lemma}{Lemma}
\newtheorem{remark}{Remark}
\newtheorem{definition}{Definition}
\newtheorem{example}{Example}
\usepackage{epsfig} 
\usepackage{psfrag}
\DeclareMathOperator{\diag}{diag}
\usepackage{tcolorbox}

\newcommand{\bal}[1] {\ensuremath{\left(\begin{array}{#1}}}
\newcommand{\ear} {\ensuremath{\end{array}\right)}}

\newcommand{\bals}[1] {\ensuremath{\left[\begin{array}{#1}}} % Begin Array Left Square
\newcommand{\ears} {\ensuremath{\end{array} \right] }} % End Array Right Square
\DeclareMathOperator{\trace}{tr}

\let\leq\leqslant
\let\geq\geqslant

\newcommand{\bmat}{\begin{matrix}}
\newcommand{\emat}{\end{matrix}}
\newcommand{\bbm}{\begin{bmatrix}}
\newcommand{\ebm}{\end{bmatrix}}
\newcommand{\bpm}{\begin{pmatrix}}
\newcommand{\epm}{\end{pmatrix}}
\newcommand{\bse}{\begin{subequations}}
\newcommand{\ese}{\end{subequations}}
\newcommand{\beq}{\begin{equation}}
\newcommand{\eeq}{\end{equation}}
\newcommand{\ben}{\begin{enumerate}}
\newcommand{\een}{\end{enumerate}}
\newcommand{\beni}{\renewcommand{\labelenumi}{\roman{enumi}.}
\renewcommand{\theenumi}{\roman{enumi}}\begin{enumerate}}
\newcommand{\eeni}{\end{enumerate}\renewcommand{\labelenumi}{\arabic{enumi}.}
\renewcommand{\theenumi}{\arabic{enumi}}}
\newcommand{\bena}{\renewcommand{\labelenumi}{\alpha{enumi}.}
\renewcommand{\theenumi}{\alpha{enumi}}\begin{enumerate}}
\newcommand{\eena}{\end{enumerate}\renewcommand{\labelenumi}{\arabic{enumi}.}
\renewcommand{\theenumi}{\arabic{enumi}}}
\newcommand{\bit}{\begin{itemize}}
\newcommand{\eit}{\end{itemize}}

\usepackage{enumerate}
\usepackage{cite,graphicx,amssymb,amsfonts,booktabs,multirow,array,comment}

\usepackage{algorithm}
\usepackage{algpseudocode}

\begin{document}
\title{\LARGE   Attack Resilient Interconnected Second Order Systems: \\A Game-Theoretic Approach}

\author{Mohammad Pirani, Joshua A. Taylor, and Bruno Sinopoli
\thanks{M. Pirani and J. A. Taylor are with the Department of Electrical and Computer Engineering, University of Toronto. E-mail: {\tt \{pirani\}@kth.se}, {\tt \{josh.taylor\}@utoronto.ca}. Bruno Sinopoli is with the Department of Electrical and Systems Engineering, Washington University in St. Louis E-mail: {\tt \{bsinopoli\}@wustl.edu}.}%
}

\newcommand{\tb}{\color{blue}}
\newcommand{\tr}{\color{red}}
\newcommand{\tg}{\color{green}}

\maketitle

\thispagestyle{empty}
\pagestyle{empty}

%%%%%%%%%%%%%%%%%%%%%%%%%%%%%%%%%%%%%%%%%%%%%%%%%%%%%%%%%%%%%%
\begin{abstract}
This paper studies the resilience of second-order networked dynamical systems to strategic attacks. We discuss two widely used control laws, which have applications in power networks and formation control of autonomous agents. In the first control law, each agent receives pure velocity feedback from its neighbor. In the second control law, each agent receives its velocity relative to its neighbors. The attacker selects a subset of nodes in which to inject a signal, and its objective is to maximize the  $\mathcal{H}_2$  norm of the system from the attack signal to the output. The defender improves the resilience of the system by adding self-feedback loops to certain nodes of the network with the objective of minimizing the system's $\mathcal{H}_2$  norm. Their decisions comprise a strategic game. Graph-theoretic necessary and sufficient conditions for the existence of Nash equilibria are presented. In the case of no Nash equilibrium, a Stackelberg game is discussed, and the optimal solution when the defender acts as the leader is characterized. For the case of a single attacked node and a single defense node, it is shown that the optimal location of the defense node for each of the control laws is determined by a specific network centrality measure. The extension of the game to the case of multiple attacked and defense nodes is also addressed. 

\end{abstract}

\section{Introduction}

\subsection{Motivation}
The resilience of cyber-physical systems to strategic attacks is one of the primary concerns in the design level and real-time operation of interconnected systems. Examples of such systems include power networks,  water and gas networks, and transportation systems. A subtle difference between faults and attacks is that in the latter, the attacker uses knowledge of vulnerabilities  to maximize its effect and/or minimize its visibility or effort to attack.  The defender thus has to adopt an intelligent strategy to counter the attacker. One approach to modeling interactions between intelligent attackers and defenders is via game theory. 

\subsection{Related Work}

Security and resilience of cyber-physical systems from the game-theoretic standpoint has attracted attention in the past decade; see \cite{zhubasar, Manshaie, Cedricbasar, Amin2013, Aminn, Piranigame } and references therein. The notion of {\it games-in-games} in cyber-physical systems reflects two  interconnected games, one in the cyber layer and the other in the physical layer, for which the payoff of each game affects the result of the other one \cite{Basarz}. Some approaches discussed appropriate game strategies, e.g., Nash or Stackelberg, based on the type of adversarial behavior (active or passive) \cite{Cedricbasar,EPFL}. The evolution of  networked systems are modeled as cooperative games \cite{Mardenshamma} and the resilience of these games to adversarial actions and/or communication failures are investigated \cite{Brown, Shankar}. There is a large literature on the security of first and second order systems \cite{dibaji2017resilient, shames, Hogan}. To date, no approach uses game theory to model the actions of intelligent attackers and defenders in second order  systems.

\subsection{Contributions}
The contributions of this paper are as follows:

\begin{itemize}
 \item We discuss an attacker-defender game on the resilience of two canonical forms of second order systems.  The attacker targets a set of nodes in the network to maximize the system $\mathcal{H}_2$ norm from the attack signals to the output, while the defender  chooses a set of nodes (to install feedback control) in order to minimize this system norm (or mitigate the effect of the attack). 
 
 \item Necessary and Sufficient conditions for the existence of Nash equilibrium (NE) for the game for each of the two second-order dynamics is discussed (Propositions \ref{prop:necsufgame1} and  \ref{prop:none}). For the cases where there is no NE, a Stackelberg game is discussed when the defender acts as the game leader (Theorems \ref{thm:necsuefgame} and \ref{thm:complexx} and Corollary \ref{cor:costfew}). 
 
 \item  For the case of a single attacked node and a single defense node, it is shown that the optimal location of the defense node in the network for each of the second order systems introduces a specific network centrality measure (Remark \ref{rem:ssodjnf}).

 \item The extension of the game to the case of multiple attacked and defense nodes is also addressed (Theorems \ref{thm:multistack}, and \ref{thm:complex}).\footnote{ Proofs of all the results of this paper are placed in the Appendix.}
\end{itemize}

It worth noting that for resilient distributed control algorithms proposed in the literature,  a large level of network connectivity is  required to bypass the effects of malicious actions \cite{Hogan, AntonioBicchi}. However, in many real-world applications, e.g., power systems, the underlying topology is designed and can not be changed. From this view, the defense mechanism proposed in this paper has an advantage compared to the previous methods in the sense that it does not rely on the connectivity level of the underlying network.

\section{Graph Theory}
\label{sec:not}
We use $\mathcal{G}=\{\mathcal{V},\mathcal{E}\}$ to denote an undirected graph where $\mathcal{V}$ is the set of vertices (or nodes) and   $\mathcal{E}\subseteq \{\{v_i,v_j\}|v_i,v_j\in \mathcal{V}, i\neq j \}$ is the set of undirected edges, where  $e=\{v_i,v_j\}\in \mathcal{E}$ if an only if there exists an undirected edge between $v_i$ and $v_j$.  Let $n=|\mathcal{V}|$. The adjacency matrix of $\mathcal{G}$ is denoted  $A$, where $A_{ij}=1$ if  there is an edge between $v_j$ and $v_i$ in $\mathcal{G}$ and zero otherwise. The {\it neighbors} of vertex $v_i \in \mathcal{V}$ in the graph $\mathcal{G}$ are denoted by the  set $\mathcal{N}_i = \{v_j \in \mathcal{V}~|~\{v_j, v_i\} \in \mathcal{E}\}$. We define the degree  for node $v_i$ as $d_i=\sum_{v_j\in \mathcal{N}_i} A_{ij}$.  The Laplacian matrix of an undirected  graph is denoted by $L = D - A$, where $D = \diag (d_1, d_2, ..., d_n)$. We use $\mathbf{e}_i$ to indicate the $i$-th vector of the canonical basis. The eccentricity $\epsilon (v)$ of a vertex $v$ in a connected graph $\mathcal{G}$ is the maximum graph distance  between $v$ and any other vertex $u \in \mathcal{G}$. The center of a graph is a set of vertices with minimum  eccentricity.   The {\it effective resistance} between a pair of nodes $i$ and $j$, denoted  $R_{ij}$, is the electrical resistance measured across nodes  $i$ and $j$ when the network represents an electrical circuit where each edge $e$ has electrical conductance  $w_e$ \cite{Ghosh2008}. The effective eccentricity $\epsilon_f (v)$ of a vertex $v$ in a connected graph $\mathcal{G}$ is the maximum graph effective resistance between $v$ and any other vertex $u$ of $\mathcal{G}$. The {\it effective center} of a graph is a set of vertices with minimum effective eccentricity.  A degree central node in the network is the node with the largest degree.

\section{System Model and Preliminaries}

Consider a network of  agents $\mathcal{V}$ where each agent follows second-order dynamics
\begin{align}
\dot{x}_i(t)&=v_i(t),\nonumber\\
\dot{v}_i(t)&=u_i(t)+w_i(t),
\label{eqn:111}
\end{align}
where $x_i(t)$ and $v_i(t)$ represent position (or phase) and velocity (or frequency), respectively.  $u_i(t)$ and $w_i(t)$ are the control input and additive disturbance to the dynamics. The control policy can be either of the following two 
 \begin{subequations}
\begin{align}
 u_i&=-\sum_{j\in \mathcal{N}_i}a_{ij}(x_i-x_j)-(b_i+b_0)v_i.\label{subeq2}
\\
u_i&=-\sum_{j\in \mathcal{N}_i}a_{ij}(x_i-x_j)-\sum_{j\in \mathcal{N}_i}b_{ij}(v_i-v_j)-a_0x_i-b_0v_i.
\label{subeq1}
%\label{eqn:laww}
\end{align}
\end{subequations}
Here $a_{ij}, b_{ij}, a_0$ and $b_0$ are nonnegative control gains. 
Control law \eqref{subeq2} uses the {\it relative position} and {\it absolute velocity} as feedbacks whereas \eqref{subeq1} uses both relative position and velocity as control feedbacks. To simplify our analysis, we assume that  $a_{ij}=b_{ij}=1$  and $a_0=b_0=k>0$, where $k$ is called the defender's control gain.\footnote{This parameter is private and only known by the system designer.}  Note that all of the analysis in this paper can be readily extended to the weighted case. Control laws \eqref{subeq2} and \eqref{subeq1} are canonical forms of well-known second-order systems. In particular, \eqref{subeq2} is the linearized swing equation for a network of power generators \cite{swing, PiraniJohnCDC}, and  \eqref{subeq1} describes the formation control of autonomous agents, e.g., a platoon of connected vehicles \cite{atkins}.

\subsection{Attack Model}
 Let $\mathfrak{F}$ denote the set of nodes under attack. The state of a node which is under attack evolves as 
\begin{align}
\dot{x}_i(t)&=v_i(t)+\zeta_{1i}(t),\nonumber\\
\dot{v}_i(t)&=u_i(t)+w_i(t)+\zeta_{2i}(t), \quad i \in \mathfrak{F},
\label{eqn:aodvjn}
\end{align}
where $\zeta_{1i}(t)$ and $\zeta_{2i}(t)$ are the effects of attack signals on the first and the second states, respectively. In vector form, \eqref{eqn:aodvjn} is given by
\begin{align}
\dot{\boldsymbol{X}}={A}\boldsymbol{X}
     +{B}_1\boldsymbol w(t)
     +{B}_2\boldsymbol \zeta(t),
     \label{eqn:powernetwork}
\end{align}
where $\boldsymbol{X}=[\boldsymbol x \quad \dot{\boldsymbol x}]^T$, $\boldsymbol{\zeta}=[\boldsymbol \zeta_1 \quad {\boldsymbol \zeta_2}]^T$. Depending on whether control law \eqref{subeq2} or \eqref{subeq1} is in place, the matrix $A$ respectively takes on the form
\begin{equation}
A=\begin{bmatrix}
       \mathbf{0}_{n} & I_{n}         \\[0.3em]
     -{L} & -H
     \end{bmatrix}, \hspace{0.6 mm} {\rm or} \hspace{1 mm} A=\begin{bmatrix}
       \mathbf{0}_{n} & I_{n}         \\[0.3em]
  -\bar{L} & -\bar{L}
     \end{bmatrix}, 
     \label{eqn:dy}
\end{equation}
where $\bar{L}\triangleq L+kD_y$ and $H\triangleq I_n+kD_y$ where $D_y=\diag(\boldsymbol y)$. $\boldsymbol y$  is a binary vector, i.e.,  $y_i\in \{0,1\}$, whose $i$-th element is one if node $i$ has a self feedback and zero if it does not. We assume that $D_y$ has at least one nonzero diagonal element so that $\bar{L}$ is non-singular \cite{ArxiveRobutness}. Here  $B_1=[\mathbf{0} \quad I_n]^T$ and $B_2=I_2\otimes F$, since we assume that if node $i$ is under attack, then its both states are affected by the attack signal.
Matrix $B_2$ encodes the decisions of the attacker. The $i$-th row of $F$ has a single 1 if node $i$ is affected by the attack, and all zeros otherwise.  
The set of nodes under attack and the set of nodes with feedback (defense nodes) are denoted by $\mathfrak{F}$ and $\mathfrak{D}$, respectively. An example of attacker and defender actions on a networked system is schematically shown in Fig.~\ref{fig:s1} (a).

\subsection{Attacker-Defender Game}
 Because we do not have a priori knowledge of the frequency contents of the attack signal, we must choose a system norm which captures the average impact of all frequencies of the attack input. We therefore choose system $\mathcal{H}_2$ norm, which is widely used to measure the level of coherence in synchronization of coupled oscillators \cite{Bamieh2, Tegling}. We first calculate the $\mathcal{H}_2$ norm of \eqref{eqn:powernetwork}. 

\begin{proposition}\label{prop:h2normsdouble}
The $\mathcal{H}_2$ norms of \eqref{eqn:powernetwork}  from the attack signal $\boldsymbol{\zeta}(t)$ to output $\boldsymbol y=\dot{\boldsymbol x}$  for control laws \eqref{subeq2} and \eqref{subeq1} are
\begin{align}
||G_1||^2_2&=\frac{1}{2}\sum_{i\in \mathfrak{F}}H_{ii}^{-1}d_i+
\frac{1}{2}\sum_{i\in \mathfrak{F}}H_{ii}^{-1},\nonumber\\
||G_2||^2_2&=\frac{1}{2}f+
\frac{1}{2}\sum_{i\in \mathfrak{F}}\bar{L}_{ii}^{-1},
\label{eqn:biaubv}
\end{align}
 where $f$ is the number of attacked nodes, $G_1$ and $G_2$ are transfer functions of  \eqref{eqn:powernetwork} from $\boldsymbol{\zeta}(t)$ to  $\dot{\boldsymbol x}$  for control laws \eqref{subeq2} and \eqref{subeq1}, respectively. $H_{ii}^{-1}$ and $\bar{L}_{ii}^{-1}$ are the $i$-th diagonal elements of $H^{-1}$ and $\bar{L}^{-1}$, respectively.
\end{proposition}
Now, we formally define the attacker-defender game.
\begin{definition}[\textbf{Attacker-Defender Game}]
The attacker chooses a set of $f$ nodes to attack, $\mathfrak{F}\subseteq \mathcal{V}$, in order to maximize the  $\mathcal{H}_2$ norm from the attack signal $\boldsymbol \zeta (t)$ to the output $\boldsymbol y=\dot{\boldsymbol x}$. The defender places local feedback control at $f$ nodes, $\mathfrak{D}\subseteq \mathcal{V}$, to minimize the system $\mathcal{H}_2$ norm.\footnote{Due to the lack of knowledge of the number of attack signals,  the defender considers $f$ as an upper bound of the number of attacked nodes and acts based on this worst-case scenario.} The result is a zero-sum game in which the payoff, based on \eqref{eqn:biaubv}, is given by
\begin{align}\label{eqn:objectives}
J_1(\mathfrak{F}, \mathfrak{D})&=\frac{1}{2}\sum_{i\in \mathfrak{F}}H_{ii}^{-1}d_i
+
\frac{1}{2}\sum_{i\in \mathfrak{F}}H_{ii}^{-1},\nonumber\\
J_2(\mathfrak{F}, \mathfrak{D})&=\frac{1}{2}f+
\frac{1}{2}\sum_{i\in \mathfrak{F}}\bar{L}_{ii}^{-1}.
\end{align}
The set of attacked nodes $\mathfrak{F}$ determine matrix $B_2$,  and the set of defense nodes   $\mathfrak{D}$ determines matrix $D_y$  and consequently matrices $H$ and $\bar{L}$  in \eqref{eqn:dy}.
\end{definition}
\iffalse
\begin{tcolorbox}
\textbf{Attacker-Defender Game:} The attacker chooses a set of $f$ nodes to attack, $\mathfrak{F}\subseteq \mathcal{V}$, in order to maximize the  $\mathcal{H}_2$ norm from the attack signal $\boldsymbol \zeta (t)$ to the output $\boldsymbol y=\dot{\boldsymbol x}$. The defender places local feedback control at $f$ nodes, $\mathfrak{D}\subseteq \mathcal{V}$, to minimize the system $\mathcal{H}_2$ norm.\footnote{Due to the lack of knowledge of the number of attack signals,  the defender considers $f$ as an upper bound of the number of attacked nodes and acts based on this worst-case scenario.} The result is a zero-sum game in which the payoff, based on \eqref{eqn:biaubv}, is given by
\begin{align}\label{eqn:objectives}
J_1(\mathfrak{F}, \mathfrak{D})&=\frac{1}{2}\sum_{i\in \mathfrak{F}}H_{ii}^{-1}d_i
+
\frac{1}{2}\sum_{i\in \mathfrak{F}}H_{ii}^{-1},\nonumber\\
J_2(\mathfrak{F}, \mathfrak{D})&=\frac{1}{2}f+
\frac{1}{2}\sum_{i\in \mathfrak{F}}\bar{L}_{ii}^{-1}.
\end{align}
The set of attacked nodes $\mathfrak{F}$ determine matrix $B_2$,  and the set of defense nodes   $\mathfrak{D}$ determines matrix $D_y$  and consequently matrices $H$ and $\bar{L}$  in \eqref{eqn:dy}.
\end{tcolorbox}
\fi

The actions of the attacker and the defender, when $f$ nodes are under attack and ${f}$ nodes are defended,  define a matrix game $\mathcal{M}_{\binom{n}{{f}}\times \binom{n}{{f}}}$. Here $\mathcal{M}_{ij}=J(\mathfrak{F}_j,\mathfrak{D}_i)$, where $\mathfrak{F}_j$ corresponds to the set chosen by the attacker and $\mathfrak{D}_i$ corresponds to the set chosen by the defender. In other words, the attacker,  the maximizer, chooses columns of matrix $\mathcal{M}$ and the defender, the minimizer, chooses the rows.

\begin{figure}[t!]
\centering
\includegraphics[scale=.35]{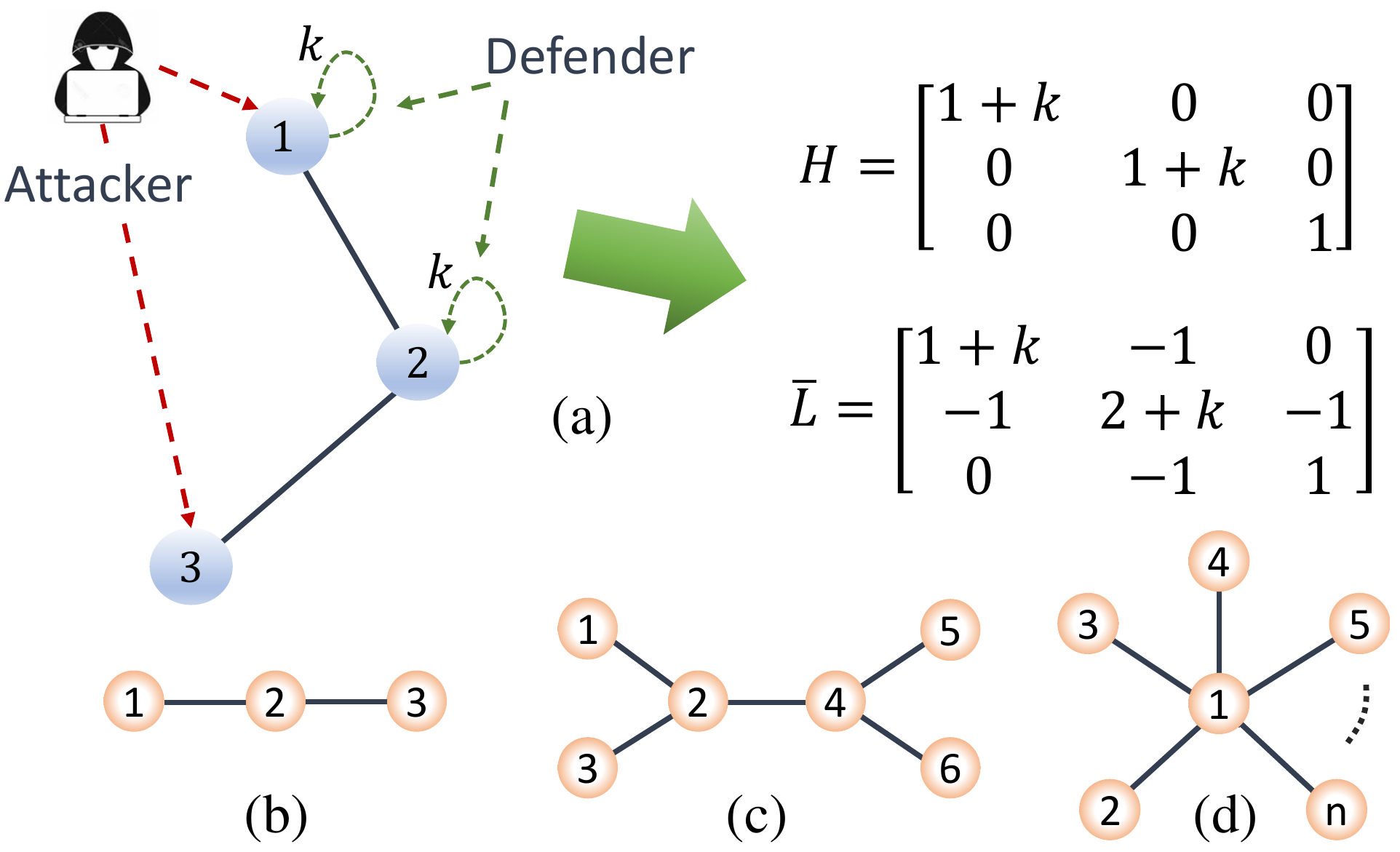}
\caption{(a) An example of an attacker-defender game and matrices $H$ and $\bar{L}$, (b) graph topology discussed in example \ref{exp:alfdkf}, (c) a graph structure where $\Delta_1=\Delta_2$, thus does not admit NE, (d) star graph admits largest threshold for $k$ over all connected graphs as $\Delta_1=n-1$ and $\Delta_2=1$. }
\label{fig:s1}
\end{figure}

\section{Attacker-Defender Game on $J_1(\mathfrak{F}, \mathfrak{D})$}

In this section, we discuss  equilibrium strategies for the attacker-defender game when the control law is \eqref{subeq2}. First, consider a single attacked node and single defense node. 

\subsection{Single Attacked-Single defense Nodes}

In this case, attacker's payoff is 
\begin{align}\label{eqn:objectsives}
J(\mathfrak{F}, \mathfrak{D})=\frac{H_{ii}^{-1}}{2}\left( d_i+1 \right), \quad i\in \mathfrak{F}.
\end{align}
A Nash equilibrium may not exist, as discussed in the following example. 
\begin{example}
For the path graph of length 3 shown in Fig.~\ref{fig:s1} (b), payoff matrix becomes
\begin{align}
\mathcal{M}=\frac{1}{2}\begin{bmatrix}
      \frac{2}{k+1} & 3 & 2        \\[0.3em]
     2 &  \frac{3}{k+1} & 2 \\[0.3em]
     2 & 3 & \frac{2}{k+1}
     \end{bmatrix},
     \label{eqn:graminnn}
\end{align}
where the attacker (maximizer) chooses columns and the defender (minimizer) chooses the rows. For $k\leq \frac{1}{2}$ both the attacker and defender choose node 2 at NE, and the equilibrium payoff is
 $J^*=\frac{3}{2k+2}$. For $k$ bigger than this threshold, there is no NE for the game.
\label{exp:alfdkf}
\end{example}

The following is a necessary and sufficient condition for the existence of an NE for the attacker-defender game. 
\begin{proposition}\label{prop:necsufgame1}

Suppose that in the game on $J_1(\mathfrak{F}, \mathfrak{D})$ in  \eqref{eqn:objectives}, there are one attacked and one defense nodes.
Then there exists an NE if and only if $k\leq \frac{\Delta_1-\Delta_2}{\Delta_2+1}$, where $\Delta_1$ and $\Delta_2$ are the largest and second largest degrees of nodes in graph $\mathcal{G}$. In this case, the game value is $J^*=\frac{\Delta_1+1}{2k+2}$ and the NE strategy is that both attacker and defender choose the node(s) with the largest degree.
\end{proposition}

\begin{remark}
According to Proposition \ref{prop:necsufgame1}, the value of $k$ which ensures the existence of NE is limited by the gap between the largest and the second largest degrees in the network. For the cases where this does not hold, e.g., when the node with the largest degree is not unique as in Fig.~\ref{fig:s1} (c), there is no NE. Moreover,  the largest possible threshold for graphs on $n$ vertices corresponds to the star graph in which the threshold becomes $\frac{n-2}{2}$, as in Fig.~\ref{fig:s1} (d).
\end{remark}

When there is no NE, we instead analyze a Stackelberg game in which the defender acts as the leader. We can write $J_1(\mathfrak{F},\mathfrak{D})$ in \eqref{eqn:objectives} as
 $$J_1(\mathfrak{F},\mathfrak{D})=\frac{1}{2}\trace \left(F^T({L}+I)H^{-1}F \right).$$
As leader, the defender solves the following optimization problem
 \begin{equation}\label{eqn:stackd}
J^*(D_y)=\min_{D_y}\frac{1}{2}\trace \left(F^{*^T}(D_y)({L}+I)H^{-1}F^{*^T}(D_y) \right)
\end{equation}
where $D_y$ is chosen over all $f$ defense nodes in $\mathcal{V}$. $F^*(D_y)$ is the best response of the attacker when the strategy of the defender is $D_y$, i.e., $F^*(D_y)$ is the solution of the following optimization problem
 \begin{equation}\label{eqn:stack}
F^*(D_y)=\arg\max_F\frac{1}{2}\trace \left(F^T({L}+I)H^{-1}F \right),
\end{equation}
where $F$ is chosen over all ${f}$ attacked nodes in $\mathcal{V}$.
 Unlike NE,  a Stackelberg game always admits an equilibrium strategy.
 \begin{remark}
We  note that for the attacker to play the Stackelberg game, i.e., find the optimal strategy \eqref{eqn:stack}, it is not  necessary to know the exact value of the feedback gain $k$. According to proposition \ref{prop:necsufgame1} and Theorem \ref{thm:multistack}, which comes later, it is sufficient for the attacker to only know if $k$ is above or below the threshold $\frac{\Delta_1-\Delta_2}{\Delta_2+1}$ in order to find its best response strategy.
\end{remark}

The following theorem,  characterizes the equilibrium of the Stackelberg game.

\begin{theorem}\label{thm:necsuefgame}
Consider a Stackelberg attacker-defender game on $J_1(\mathfrak{F}, \mathfrak{D})$ in \eqref{eqn:objectives}  in which there exists a single attacked node and single defense node, the defender as the game leader, and $k> \frac{\Delta_1-\Delta_2}{\Delta_2+1}$.  Then the equilibrium strategy corresponds to the case where the defender chooses $v=\arg\max_{i\in \mathcal{V}}d_i$, i.e., the node with the largest degree. In this case, the attacker's best response will be $\bar{v}=\arg\max_{i\in \mathcal{V}\setminus v}d_i$, i.e., the node with the second largest degree.
\end{theorem}

The following example discusses the role of the threshold $\bar{k}=\frac{\Delta_1-\Delta_2}{\Delta_2+1}$ in the attacker's strategy.

\begin{example}
For the graph shown in Fig.~\ref{fig:s2} we have $\Delta_1=3, \Delta_2=2$. Hence, the threshold is $\bar{k}=\frac{1}{3}$. The attacker's decisions are plotted with respect to the defender's best action, i.e., the node with the largest degree. For $k=0.1<\bar{k}$, the attacker's best action is the node with the largest degree (as follows from Proposition \ref{prop:necsufgame1}) and for $k=1>\bar{k}$, the attacker's best response is the node with the second largest degree (as follows from Theorem  \ref{thm:necsuefgame}). For $k=\bar{k}$, the payoff will be the same when the attacker chooses either nodes 3 or 4.
\end{example}

\begin{figure}[t!]
\centering
\includegraphics[scale=.35]{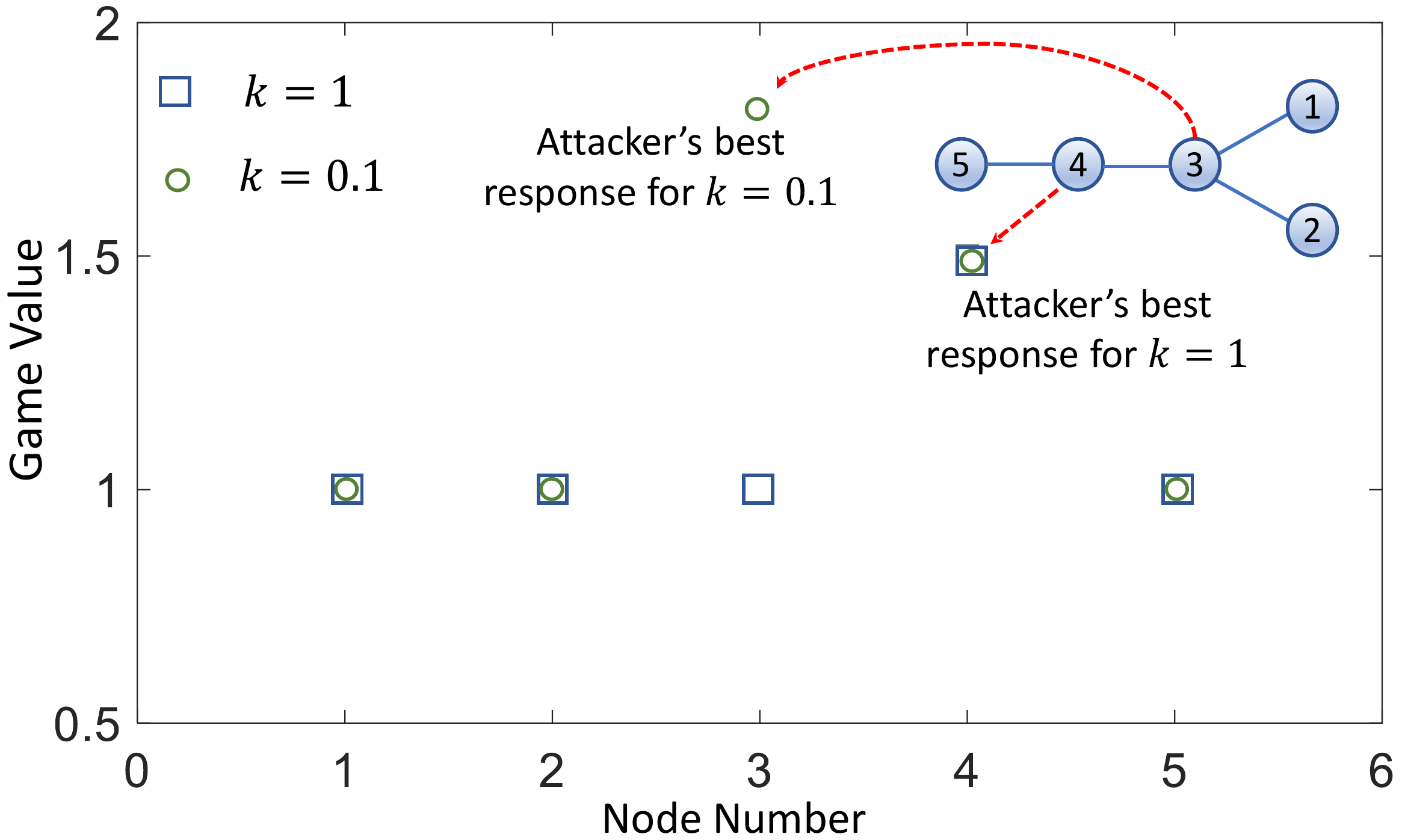}
\caption{The effect of the feedback value $k$ on attacker's best response. The defender has chosen node 3 (its optimal decision). }
\label{fig:s2}
\end{figure}
\iffalse
\begin{remark}
According to Theorem \ref{thm:necsuefgame}, as the optimal game value is $J^*=\frac{\Delta_2+1}{2}$, a larger gap between the first and the second largest degrees in the network results in a more secure system, i.e., lower system $\mathcal{H}_2$ norm. In this sense, the most secure network topology would be the star graph since in that case we have $\Delta_1=n-1$ and $\Delta_2=1$. Moreover, according to the above results, increasing the connectivity of the underlying network does not have an effect on the game value $J_1(\mathfrak{F},\mathfrak{D})$ in \eqref{eqn:objectives}.
\end{remark}
\fi

\subsection{Multiple Attacked-Multiple Defense Nodes}

Now consider the case that there exist $f$ attacked nodes and $f$ defense nodes, i.e., $|\mathfrak{F}|=|\mathfrak{D}|=f\geq 1$. Here we only consider a Stackelberg setup as it is more applicable to security problems \cite{Manshaie}. We remark that if the defender is the leader, it reflects the defender's need to consider the worst case. Thus, it is more convenient to have the defender as the game leader.  

The Stackelberg game is a combinatorial problem. Thus, in general, its computational cost would be high, unless it is reduced with specific assumptions.  With this in mind, in our problem, finding optimal defense nodes when the defender is the game leader is burdensome, unless the control gain $k$ is sufficiently large and the number of attacked nodes is sufficiently small. 

\begin{theorem}\label{thm:multistack}
Consider a Stackelberg attacker-defender game on $J_1(\mathfrak{F}, \mathfrak{D})$ in \eqref{eqn:objectives} where there exists $f$ attacked nodes and $f$ defense nodes, $f \geq 1$ and $n\geq 2f$, with the defender as the game leader.  If 
$k\geq \frac{1}{2}(fd_{\rm \max}-2)$ then at the equilibrium the defender chooses $v=\arg\max_{\substack{\mathfrak{D} \subseteq \mathcal{V}\\|\mathcal{D}|=f}} \sum_{v_i \in \mathfrak{D}}d_i$, i.e., $f$ nodes with the largest degrees in the network. The best response of the attacker is to choose $\bar{v}=\arg\max_{\substack{\mathfrak{F}\subseteq \mathcal{V}\setminus v\\|\mathcal{F}|=f}} \sum_{v_i \in \mathfrak{F}}d_i$.
\end{theorem}

\iffalse
It worth noting that, if the attacker is the leader of the Stackelberg game, without further assumption on the control gain $k$, the problem is computationally tractable. 

\begin{theorem}
Consider a Stackelberg attacker-defender game on $J_1(\mathfrak{F}, \mathfrak{D})$ in \eqref{eqn:objectives} with $f$ attacked nodes and $f$ defense nodes, $f \geq 1$, with the attacker as the game leader. Then at the equilibrium the attacker chooses  $v=\arg\max_{\mathfrak{F} \subseteq \mathcal{V}} \sum_{v_i \in \mathfrak{F}}d_i$, i.e., the $f$ nodes with largest degree in the network. In this case, the best response by the defender will be identical to the attacker.
\label{thm:skjg}
\end{theorem}
\fi

\section{Attacker-Defender Game on $J_2(\mathfrak{F}, \mathfrak{D})$}
\label{sec:oegb}
In this section, we discuss the dynamics with control law \eqref{subeq1} and objective function $J_2(\mathfrak{F}, \mathfrak{D})$ in \eqref{eqn:objectives}.

\subsection{Single Attacked, Single Defense Nodes}

Similar to the case of $J_1(\mathfrak{F}, \mathfrak{D})$, we start with the case of  single attacked and  single defense nodes. We first have the following proposition.
\begin{proposition}
The  attacker-defender game on $J_2(\mathfrak{F}, \mathfrak{D})$ in \eqref{eqn:objectives}  with a single attack and single defense node does not admit an NE. \label{prop:none}
\end{proposition}

Similar to the attacker-defender game on $J_1(\mathfrak{F}, \mathfrak{D})$, in the absence of NE, an optimal defense strategy can be determined by finding the solution of the Stackelberg game.
Recalling the notion of the {\it effective center of a graph}, from Section \ref{sec:not}, we have the following theorem which is proven in Appendix \ref{sec:thm44}.
\begin{theorem}\label{thm:complexx}
Consider the Stackelberg attacker-defender game on $J_2(\mathfrak{F}, \mathfrak{D})$ in \eqref{eqn:objectives} on graph $\mathcal{G}$ with the defender as the game leader.  Then, a solution of the game corresponds to the case when the defender chooses the effective center of  $\mathcal{G}$, i.e., $\mathfrak{D}^*=\arg\min_{v \in \mathcal{V}}\epsilon_f (v)$. In this case, the best response of the attacker will be $\mathcal{B}^*(\mathfrak{D})=\arg\max_{j\in \mathcal{V}}R_{\mathfrak{D}^*j}$, i.e., a node with the maximum effective resistance from $\mathfrak{D}^*$.
\end{theorem}

For the case of acyclic networks, Theorem \ref{thm:complexx} reduces to the following corollary.

\begin{corollary}[\textbf{Acyclic Networks}]\label{cor:costfew}
Consider the Stackelberg attacker-defender game on $J_2(\mathfrak{F}, \mathfrak{D})$ in \eqref{eqn:objectives}, with the defender as the game leader, over the connected undirected tree $\mathcal{G}$.  At equilibrium,  the defender chooses the  {\it center} of the graph and the attacker chooses the node with the greatest distance from the center. 
\end{corollary}

\begin{remark}[\textbf{Game Equilibriums and Network Centrality}]\label{rem:ssodjnf}
As mentioned before, the optimal location of the defense node for the objective function $J_1(\mathfrak{F}, \mathfrak{D})$ is the degree central node (Theorem \ref{thm:necsuefgame})  and for $J_2(\mathfrak{F}, \mathfrak{D})$ is the graph's center for acyclic networks (Corollary \ref{cor:costfew}) or effective center for general graphs (Theorem \ref{thm:complexx}). These network centralities (and consequently optimal defense node placements) can differ substantially from each other. One of such examples is the graph shown in Fig.~\ref{fig:center1} (a), in which by increasing the length of the  path, the two centralities become far apart.
\end{remark}

\begin{figure}[t!]
\centering
\includegraphics[scale=.4]{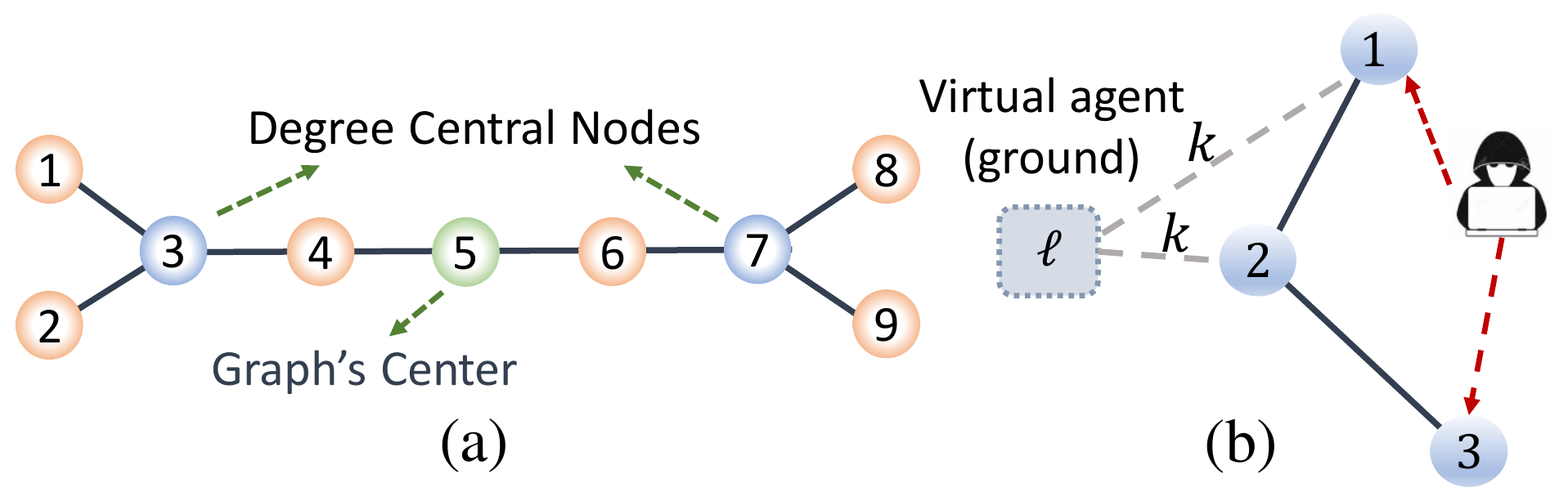}
\caption{(a) Optimal locations of the defense node for objective functions $J_1(\mathfrak{F}, \mathfrak{D})$ and $J_2(\mathfrak{F}, \mathfrak{D})$, (b) Extended graph and the virtual agent (ground). }
\label{fig:center1}
\end{figure}

\subsection{Multiple Attacked, Multiple defense Nodes}\label{sec:effectaafdd}

In order to tackle this problem, we interpret the self-feedback loops in the form of connections to  some virtual agent (or grounded node) as shown in Fig. \ref{fig:center1} (b). In this case, matrix $\bar{L}$ would be  a submatrix of the Laplacian matrix $L_{(n+1)\times (n+1)}$ of the extended graph (including $\ell$) where the row  and the column corresponding to $\ell$ are removed. Such submatrices are called grounded Laplacian in the literature \cite{ArxiveRobutness}. With this is mind, it is known that the $i$-th diagonal element of $\bar{L}^{-1}$ is $R_{i\ell}$, i.e., the effective resistance between node $v_i$ and the virtual node $\ell$ \cite{Ghosh2008}.\footnote{When the graph is a tree, the effective resistance and physical distance become the same.}
As an example, consider nodes 1 and 2 in Fig.~\ref{fig:center1} (b) which are chosen as defenders and nodes 1 and 3 which are under attack. In this case, we have $J_2(\mathfrak{F},\mathfrak{D})=1+\frac{1}{2}\sum_{i\in \mathfrak{F}}\bar{L}^{-1}_{ii}=1+\frac{1}{2}(R_{1\ell}+R_{3\ell})$. Based on this fact, the proof of the  following theorem is straightforward.
\begin{theorem}\label{thm:complex}
Consider the Stackelberg attacker-defender game on $J_2(\mathfrak{F}, \mathfrak{D})$ in \eqref{eqn:objectives} with $f$ defense nodes and  $f$ attack nodes, $f\geq 1$, with the defender as the game leader, over the connected undirected graph $\mathcal{G}$.  Denote the virtual agent corresponding to a set of $f$ defense nodes $\mathfrak{D}$ by $\ell_{\mathfrak{D}}$. Then, a solution of the game is when the defender chooses set $\mathfrak{D}$ in which the maximum sum of effective resistances between $\ell_{\mathfrak{D}}$ and all combinations of $f$ nodes in the network is minimized, i.e., $\mathfrak{D}^*=\arg\min_{\mathfrak{D}\subseteq \mathcal{V}}\max_{\mathfrak{F}\subseteq \mathcal{V}}\sum_{j\in \mathfrak{F}}R_{\ell_{\mathfrak{D}}j}$. In this case, the attacker chooses the set of $f$ attacked nodes as $\mathfrak{F}^*=\arg\max_{\mathfrak{F}\subseteq \mathcal{V}}\sum_{j\in \mathfrak{F}}R_{\ell_{\mathfrak{D}^*}j}$. 
\end{theorem}

As it is seen from Theorem \ref{thm:complex}, finding the optimal set of defense nodes  requires a high level of computation. 

\begin{remark}\textbf{(The Effect of Increasing Connectivity)}:
Since the effective resistance between two nodes in the graph is an increasing function of edge weights  \cite{ArxiveRobutness}, adding extra edges to the network (or increasing the weight of edges)  decreases the diagonal elements of $\bar{L}^{-1}$ and consequently decreases the system $\mathcal{H}_2$ norm. Hence, unlike control law \eqref{subeq2}, increasing connectivity is beneficial from the defender's perspective for \eqref{subeq1}.
\label{rem:1s}
\end{remark}

\begin{appendix}

\subsection{Proof of Proposition \ref{prop:h2normsdouble}}
\label{sec:aewoegi}

\begin{proof}
We prove for the first case, the second case \eqref{subeq1} follows a similar procedure. We compute the $\mathcal{H}_2$ norm using the trace formula $||G||_2^2=\trace(B_2^{\sf T}\mathcal{W}_oB_2)$, where $\mathcal{W}_o$ is the observability Gramian $\mathcal{W}_o=\int_{0}^{\infty}e^{A^{\sf T}t}C^{\sf T}Ce^{At}$ and it is uniquely obtained from the Lyapunov equation $\mathcal{W}_o{A}+{A}^{\sf T}\mathcal{W}_o=-C^{\sf T}C$   with an additional constraint $\mathcal{W}_o v=\mathbf{0}_{2n}$ where $v=[\mathbf{1}_n^{\sf T} \hspace{1mm} \mathbf{0}_n^{\sf T}]^{\sf T}$ is the mode corresponding to the marginally stable eigenvalue of $A$. It is due to the fact that the marginally stable mode $v$ is not detectable, i.e., $Ce^{At}v=Cv=\mathbf{0}_{2n}$ for all $t \geq 0$, and since the rest of the eigenvalues are stable, the indefinite integral exists. The proof of the uniqueness of $\mathcal{W}_o$ is the same as \cite[Lemma 1]{Poola} and is omitted here. To calculate the observability Gramian, we have 
\begin{align}
\begin{bmatrix}
      W_{11} & W_{12}        \\[0.3em]
     W_{21} & W_{22}
     \end{bmatrix}A
     &+A^{\sf T}\begin{bmatrix}
      W_{11} & W_{12}        \\[0.3em]
     W_{21} & W_{22}
     \end{bmatrix} =\begin{bmatrix}
     \mathbf{0}_{n}  & \mathbf{0}_{n}       \\[0.3em]
     \mathbf{0}_{n} & -I_{n} 
     \end{bmatrix}
     \label{eqn:gramin}
\end{align}
By solving \eqref{eqn:gramin} we get $W_{11}=\frac{1}{2}LH^{-1}$, $W_{22}=\frac{1}{2}H^{-1}$ and $W_{12}=W_{21}=\mathbf{0}$. Hence we have $||G||_2^2=\trace(B_2^{\sf T}\mathcal{W}_oB_2)=\trace(F^TW_{11}F+F^TW_{22}F)$ which yields the result. 
\end{proof}

\subsection{Proof of Proposition \ref{prop:necsufgame1}}
\label{sec:aekjfnal}

\begin{proof}
It is easy to verify that each element of the matrix game $\mathcal{M}$ is 
\begin{align}
\mathcal{M}_{ij}=
  \begin{cases}
     \frac{d_j+1}{2k+2}        & \quad  i=j,\\
    \frac{d_j+1}{2}  & \quad i \neq j.\\
  \end{cases}
\label{eqn:single}
\end{align}
We first prove the sufficient condition, i.e., assume $k\leq \frac{\Delta_1-\Delta_2}{\Delta_2+1}$. Then, if the attacker changes its strategy (unilaterally) from the node with the maximum degree to some node $v_i$, according to \eqref{eqn:single} and the upper bound for $k$, the game value becomes $J=\frac{d_i+1}{2}\leq \frac{\Delta_1+1}{2k+2}$. Moreover, if the defender wants to change its strategy to another node $v_i$, based on \eqref{eqn:single}  since  the smallest element of each column is its diagonal element, it will get $J=\frac{d_i+1}{2}\geq \frac{\Delta_1+1}{2k+2}$. Hence, neither the attacker nor the defender are willing to change their strategy unilaterally.

Now suppose that having both attacker and defender choose the node with the largest degree is NE. Then we have to have $\frac{\Delta_1+1}{2k+2}\geq \frac{d_j+1}{2}$ for all $j=1,2,...,n$ which results in $k\leq \frac{\Delta_1-d_i}{d_i+1}$ for all $j=1,2,...,n$ and this proves the claim.
\end{proof}

\subsection{Proof of Theorem \ref{thm:necsuefgame}}
\label{sec:EWOGqwdUB}

\begin{proof}
When $k> \frac{\Delta_1-\Delta_2}{\Delta_2+1}$,  for each row (defender's action) of matrix $\mathcal{M}$,  the largest element (the best response of the attacker) will be $\frac{1}{2}(\Delta_1+1)$, except the row corresponding to the node with the largest degree. In that row, the largest element will be $\frac{1}{2}(\Delta_2+1)$. Since $\Delta_1 \geq \Delta_2$, the optimal action of the defender will be $v=\arg\max_{i\in \mathcal{V}}d_i$. The best response of the attacker will be the node with the second largest degree. This solution may not be unique, however, the optimal value of this game is unique and given by $J^*=\frac{\Delta_2+1}{2}$.
\end{proof}

\subsection{Proof of Theorem \ref{thm:multistack}}
\label{sec:thmssss}

\begin{proof}
For multiple attacked-multiple defense nodes case, each element of the matrix game $\mathcal{M}_{ij}$ (corresponding to defender decision set $\mathfrak{D}_i$ and attacker decision set $\mathfrak{F}_j$) is

\begin{align}
\mathcal{M}_{ij}=
  \begin{cases}
     \frac{\sum_{j\in \mathfrak{F}_j}d_k}{2k+2}+\frac{f}{2k+2}        &  i=j,\\ \\
     \frac{\sum_{k\in \mathfrak{F}_j\cap \mathfrak{D}_i}d_k}{2k+2}+\frac{\gamma_1^{ij}}{2k+2}\\
     +\frac{1}{2}(\sum_{k\in \mathfrak{F}_j\setminus \mathfrak{D}_i}d_k+\gamma_2^{ij})  &  i \neq j,\\
  \end{cases}
\label{eqn:multiple}
\end{align}
where $\gamma_1^{ij}=|\mathfrak{F}_j\cap \mathfrak{D}_i|$ and $\gamma_2^{ij}=|\mathfrak{F}_j\setminus \mathfrak{D}_i|=f-\gamma_1^{ij}$.
Since the defender is the game leader, it has to choose a row in game matrix $\mathcal{M}$ whose maximum element is minimum (over all other rows). When $k$ is lower bounded by $k\geq \frac{1}{2}(fd_{\rm \max}-2)$, considering a fixed set $\mathfrak{D}$ (set of defense nodes), for each set of attacked nodes $\mathfrak{F}$ we have
\begin{align}
\frac{\sum_{j\in\mathfrak{F}\cap \mathfrak{D}}d_j}{2k+2}+\frac{\gamma_1}{2k+2}\leq 2. \quad \forall \mathfrak{F}\subseteq \mathcal{V}
\label{eqn:ineqqq}
\end{align}
Inequality \eqref{eqn:ineqqq} together with \eqref{eqn:multiple} shows that for the row corresponding to set $\mathfrak{D}$, its largest  element  corresponds to the set of attackers  $\bar{\mathfrak{F}}$  for which $\bar{\mathfrak{F}}\cap \mathfrak{D}=\phi$. In order for this to happen, we must have $n\geq 2f$. In this case, $\gamma_1=0$ and the maximum element in the row corresponding to set $\mathfrak{D}$ is  (according to the second term in \eqref{eqn:multiple}) $\bar{\mathcal{M}}=\max_{\mathfrak{F}\subseteq \mathcal{V}\setminus \mathfrak{D}} \frac{1}{2}\sum_{j\in \mathfrak{F}}d_j+f$. Thus, the best action of the defender, to minimize that maximum row element, is to choose
$\bar{\mathfrak{D}}=\arg\max_{\mathfrak{D}\subseteq \mathcal{V}} \sum_{j\in \mathfrak{D}}d_j$, i.e., $f$ nodes with largest degrees in the graph.
\end{proof}

\iffalse
\subsection{Proof of Theorem \ref{thm:skjg}}
\label{sec:EWOGUB}
\begin{proof}
 Based on \eqref{eqn:multiple} we can easily see that $\mathcal{M}_{ij}\geq \mathcal{M}_{jj}$ for all $i=1,2,...,\binom{n}{f}$.
Since the attacker is the game leader, it chooses a column whose minimum element is maximum (over all other columns). According to \eqref{eqn:multiple}, the minimum element of column $j$ is $\mathcal{M}_{jj}$. Thus, the attacker simply chooses $\arg\max_{j}\mathcal{M}_{jj}$, which is, according to \eqref{eqn:multiple}, equivalent to a set which maximizes $\frac{\sum_{j\in \mathfrak{F}}d_j}{2k+2}$ and this proves the claim. Since the smallest element in each column is its diagonal element, then the best response of the defender will be to choose the same attacked nodes. 
\end{proof}
\fi

\subsection{Proof of Proposition \ref{prop:none}}
\label{sec:prop33}

\begin{proof}
As mentioned in Section \ref{sec:oegb}, the $j$-th diagonal element of $\bar{L}^{-1}$ is the effective resistance from node $v_j$ and the virtual node $\ell$ which is connected to the single defense node $v_i$ with an edge of weight (conductance) $k$ \cite{ArxiveRobutness}. Thus, we have $ [\bar{L}^{-1}]_{jj}=R_{\ell j}$. Hence, the value of each diagonal element of the game matrix $\mathcal{M}$ is $\mathcal{M}_{ii}=\frac{1}{2}+\frac{1}{2k}$ and each off-diagonal element is $\mathcal{M}_{ij}=\frac{1}{2}+\frac{1}{2k}+ \frac{1}{2} R_{ij}$. Thus, each diagonal element is strictly less than the elements of its corresponding row and column. Now, assume that a NE exists and let $(i^*,j^*)$ denote the equilibrium strategies of the attacker and defender. Thus, we should have $[\mathcal{M}]_{i^*j}\leq [\mathcal{M}]_{i^*j^*}\leq [\mathcal{M}]_{ij^*}$
for all $i\neq i^*$ and $j\neq j^*$. 
If element $[\mathcal{M}]_{i^*j^*}$ is a diagonal element, then the left inequality will be violated and if it is a non-diagonal element, the right inequality will be violated. 
\end{proof}

\subsection{Proof of Theorem \ref{thm:complexx}}
\label{sec:thm44}
\begin{proof}
We know that for the  game matrix $\mathcal{M}$ we have  $\mathcal{M}_{ij}=\frac{1}{2}+\frac{1}{2}R_{\ell j}$, where $v_{\ell}$ is the virtual agent connected to the defense node $v_i$ with an edge with weight $k$ and $v_j$ is the attacked node. As the defender is the leader of the Stackelberg game, it minimizes (over all rows) the maximum element of each row of $\mathcal{M}$. Thus, the optimal place for the defender is $v^*=\arg\min_{i}\max_{j}\mathcal{R}_{\ell j}$ and this is the effective center of the graph defined in Section \ref{sec:not}.  Note that this solution (strategies of the defender and attacker) may not be unique since the effective center of the network may not be a single node. However, the value of the game is unique.
\end{proof}

\iffalse
\subsection{Proof of Theorem \ref{thm:costfew}}
\label{sec:thm44}
\begin{proof}
We know that for the  game matrix $\mathcal{M}$ we have  $\mathcal{M}_{ij}=\frac{1}{2}+\frac{1}{2k}+ \frac{1}{2} |\mathcal{P}_{ij}|$. As the defender is the leader of the Stackelberg game, it minimizes (over all rows) the maximum element of each row of $\mathcal{M}$. Since the term $\frac{1}{2}+\frac{1}{k}$ is shared over all elements of $\mathcal{M}$, then the optimal place for the defender is $v^*=\arg\min_{i}\max_{j}|\mathcal{P}_{ij}|$ and this is the center of the graph, whose eccentricity is minimized.  Note that this solution (strategies of the defender and attacker) may not be unique since the center of the network may not be a single node. However, the value of the game is unique.
\end{proof}
\fi
\end{appendix}

%%%%%%%%%%%%%%%%%%%%%%%%%%%%%%%%%%%%%%%%%%%%%%%%%%%%%%%%%%%%%%%
%%%%%%%%%%%%%%%%%%%%%%%%%%%%%%%%%%%%%%%%%%%%%%%%%%%%%%%%%%%%%%%
\section{Conclusion}
 \label{sec:conclusion}
A game-theoretic approach to the resilience of two canonical forms of second-order network control systems was discussed. For the case of a single attacked node and a single defense node, it was shown that the optimal location of the defense node for each of the two second-order systems introduces a specific network centrality measure. The extension of the game to the case of multiple attacked and defense nodes was also discussed and graph-theoretic interpretations of the equilibrium of the Stackelberg game for this case was investigated. An avenue for the future work is to extend these results to  directed networks.

\bibliographystyle{IEEEtran}
\bibliography{refs}

\end{document}